\documentclass[a4paper,12pt]{article}
\usepackage{amsmath,amsthm,amssymb}
\usepackage{amsmath}
\usepackage{amsfonts}
\usepackage{amscd} 
\advance \textheight by \topskip
\advance \textheight by 1pt
\makeatother

\usepackage[all]{xy}

\allowdisplaybreaks[4]

\numberwithin{equation}{section}

\newtheorem{Theorem}{Theorem}[section]
\newtheorem*{Theorem*}{Theorem}

\newtheorem{Lemma}[Theorem]{Lemma}
\newtheorem{Proposition}[Theorem]{Proposition}
 { \theoremstyle{definition}

\newtheorem{Remark}[Theorem]{Remark} }

\title{Difference of solutions for the inversion problem of ultra-elliptic integrals}
\author{Takanori Ayano\footnote{Osaka Central Advanced Mathematical Institute, Osaka Metropolitan University, \newline \hspace{3ex} 3-3-138, Sugimoto, Sumiyoshi-ku, Osaka, 558-8585, Japan. \newline \hspace{3ex} Email: ayano@omu.ac.jp
\newline \hspace{3ex} 2020 Mathematics Subject Classification: 14H42, 14K25, 32A15. 
\newline \hspace{3ex} Key words and phrases: sigma function, hyperelliptic function, inversion problem of hyperelliptic integrals. 
}}

\date{}

\begin{document}
\maketitle

\begin{abstract}
Let $V$ be a hyperelliptic curve of genus 2 defined by $Y^2=f(X)$, where $f(X)$ is a polynomial of degree 5. 
The sigma function associated with $V$ is a holomorphic function on $\mathbb{C}^2$.  
For a point $P$ on $V$, we consider the problem to express the $X$-coordinate of $P$ in terms of the image of $P$ under the Abel--Jacobi map.   
Two meromorphic functions $f_2$ and $g_2$ on $\mathbb{C}^2$ which give solutions of this problem are known. 
Since $f_2$ and $g_2$ coincide on the zero set of the sigma function, it is expected that $f_2-g_2$ can be divided by the sigma function.  
In this paper, we decompose $f_2-g_2$ into a product of the sigma function and a meromorphic function explicitly. 
\end{abstract}

\section{Introduction}

The hyperelliptic sigma functions, which are originally introduced by F. Klein \cite{Kl1, Kl2}, are extensively studied for the last three decades (see \cite{BEL-97-1, BEL-2012, BEL-2018} and references therein). 
The hyperelliptic sigma functions play important roles in the inversion problem of hyperelliptic integrals. 
The inversion problem of hyperelliptic integrals is important in mathematical physics (see \cite{AB2019, BEL-2012, E-H-K-K-L-1, E-H-K-K-L-P-1} and references therein). 
In this paper, we consider hyperelliptic integrals of genus 2, which are called ultra-elliptic integrals. 
Several functions which give solutions of the inversion problem of ultra-elliptic integrals are introduced. 
In this paper, we study relationships among these functions. 

Let $V$ be the hyperelliptic curve of genus 2 defined by 
\[Y^2=X^5+\lambda_4X^3+\lambda_6X^2+\lambda_8X+\lambda_{10},\qquad
\lambda_i\in\mathbb{C}.\]
A basis of the vector space consisting of holomorphic 1-forms on $V$ is given by 
\[
\omega_1=-\frac{X}{2Y}{\rm d}X,\qquad
\omega_3=-\frac{1}{2Y}{\rm d}X. 
\]
Let ${\boldsymbol \omega}={}^t(\omega_1,\omega_3)$. 
Let $\operatorname{Jac}(V)$ be the Jacobian variety of $V$. 
Let us consider the Abel--Jacobi map
\[I\colon\quad V\to\operatorname{Jac}(V),\qquad P\mapsto\int_{\infty}^P{\boldsymbol \omega}.\]
Let $\sigma(u)$ with $u={}^t(u_1,u_3)\in\mathbb{C}^2$ be the sigma function associated with $V$, which is a holomorphic function on $\mathbb{C}^2$ (cf. \cite{BEL-97-1, BEL-2012}). 
Let $\sigma_i=\partial_{u_i}\sigma$, $\sigma_{ij}=\partial_{u_j}\sigma_i$, and $\wp_{ij}=-\partial_{u_i}\partial_{u_j}\log\sigma$, where $\partial_{u_i}=\frac{\partial}{\partial u_i}$. 
It is well known that the inversion problem of the map $I$ can be solved in terms of the sigma function. 
Let $f_2(u)$, $g_2(u)$, $f_5(u)$, and $g_5(u)$ be the meromorphic functions on $\mathbb{C}^2$ defined by 
\begin{gather*}
f_2(u)=-\frac{\sigma_3(u)}{\sigma_1(u)},\qquad g_2(u)=\frac{\wp_{11}(2u)}{2},\\
f_5(u)=\frac{\sigma(2u)}{2\sigma_1(u)^4},\qquad g_5(u)=\frac{\sigma_1^2\sigma_{33}-2\sigma_1\sigma_3\sigma_{13}+\sigma_3^2\sigma_{11}}{2\sigma_1^3}(u).
\end{gather*}
For $P=(X,Y)\in V$, let $v=I(P)$. 
In \cite[p.~129]{G2} and \cite[Theorem 1]{J}, the formula 
\[X=f_2(v)\]
is given. 
In \cite[Lemma 3.4]{M}, the formula 
\[X=g_2(v)\]
is given. 
In \cite[p.~128]{G} and \cite[Lemma 3.2.4]{O-98}, the formula 
\[Y=f_5(v)\]
is given. 
In \cite[p.~116]{BEL-2012} and \cite[Lemma 3.2]{AB2019}, the formula 
\[Y=g_5(v)\]
is given. 
Let $W=\big\{u\in\mathbb{C}^2\;|\;\sigma(u)=0\big\}$. 
By the Riemann vanishing theorem, the image of $V$ under the map $I$ is equal to the zero set of the sigma function in $\operatorname{Jac}(V)$. 
Since the functions $f_5$ and $g_5$ coincide on $W$, the difference $f_5-g_5$ can be divided by the sigma function. 
In \cite[Lemma 7.1]{AB2019}, by using the addition theorem for the sigma function, a decomposition of $f_5-g_5$ into a product of the sigma function and a meromorphic function on $\mathbb{C}^2$ is described explicitly. 
Since the functions $f_2$ and $g_2$ coincide on $W$, the difference $f_2-g_2$ can be divided by the sigma function. 
In \cite[Lemma 7.2]{AB2019}, in the rational limit, i.e., $\lambda_4=\lambda_6=\lambda_8=\lambda_{10}=0$, a decomposition of $f_2-g_2$ into a product of the rational limit of the sigma function and a rational function on $\mathbb{C}^2$ is described explicitly.  
In this paper, in the general case, we decompose $f_2-g_2$ into a product of the sigma function and a meromorphic function on $\mathbb{C}^2$ explicitly (Theorem \ref{2023.4.27.1}). 
To be precise, we decompose $f_2-g_2$ in the form of 
\[f_2-g_2=\sigma \frac{A}{B},\]
where $A$ and $B$ are holomorphic functions on $\mathbb{C}^2$ which are not identically equal to $0$ on $W$. 
It is a non-trivial problem to decompose $f_2-g_2$ into a product of the sigma function and a meromorphic function on $\mathbb{C}^2$ explicitly in the general case. 
We express $\wp_{11}(2u)$ in terms of $\wp_{11}(u)$, $\wp_{13}(u)$, and $\wp_{33}(u)$ by using the double-angle formula for $\wp_{11}$ and the formulae which express the higher logarithmic derivatives of the sigma function in the form of polynomials in the second and third logarithmic derivatives of the sigma function. 
Further, we eliminate $\wp_{33}(u)^k$ with $k\ge2$ from this expression of $\wp_{11}(2u)$ by using the relation between $\wp_{11}(u)$, $\wp_{13}(u)$, and $\wp_{33}(u)$, which is known as the defining equation of the Kummer surface.  
The expression of $\wp_{11}(2u)$ obtained by the above operations is described in Lemma \ref{2023.7.21.333}.    
By substituting the expressions of $\wp_{11}(u)$, $\wp_{13}(u)$, and $\wp_{33}(u)$ using the sigma function into the expression of $\wp_{11}(2u)$ in Lemma \ref{2023.7.21.333}, we can draw the sigma function from $f_2-g_2$. 

In the theory of functions, it is an important problem to give a non-trivial decomposition of a meromorphic function into a product of meromorphic functions. 
In this paper, we solve this problem for the function $f_2-g_2$.  
One of the prominent features of the sigma function is that its power series expansion around the origin begins at the Schur function and all the coefficients of the expansion are polynomials in the coefficients of the defining equation of the curve. 
In the proof of Theorem \ref{2023.4.27.1}, by using this property of the sigma function, we prove that a holomorphic function on $\mathbb{C}^2$ is not identically equal to $0$ on $W$. 
This proof implies that the above property of the sigma function is useful. 

The present paper is organized as follows. 
In Section \ref{2023.3.25.1}, we review the definition of the sigma function for the curve of genus 2. 
In Section \ref{2024.2.15}, we decompose $f_2-g_2$ into a product of the sigma function and a meromorphic function on $\mathbb{C}^2$ explicitly. 

\section{The two-dimensional sigma function}\label{2023.3.25.1}
In this section, we review the definition of the sigma function for the curve of genus 2 and give facts about it which will be used later on. 
For details of the sigma function, see \cite{BEL-97-1, BEL-2012}. 

We set
\[
M(X)=X^5+\lambda_4X^3+\lambda_6X^2+\lambda_8X+\lambda_{10},\qquad
\lambda_i\in\mathbb{C}.
\]
We assume that $M(X)$ has no multiple roots. We consider the non-singular hyperelliptic curve of genus $2$
\[
V=\big\{(X,Y)\in\mathbb{C}^2\mid Y^2=M(X)\big\}.
\]
A basis of the vector space consisting of holomorphic 1-forms on $V$ is given by 
\[
\omega_1=-\frac{X}{2Y}{\rm d}X,\qquad
\omega_3=-\frac{1}{2Y}{\rm d}X. 
\]
We consider the following meromorphic 1-forms on $V$:  
\[
\eta_1=-\frac{X^2}{2Y}{\rm d}X,\qquad
\eta_3=\frac{-\lambda_4X-3X^3}{2Y}{\rm d}X, 
\]
which are holomorphic at any point except $\infty$. 
Let $\{\mathfrak{a}_i,\mathfrak{b}_i\}_{i=1}^2$ be a canonical basis in the one-dimensional homology group of the curve $V$. We define the period matrices by
\[
2\omega'=\left(\int_{\mathfrak{a}_j}\omega_i\right), \quad 2\omega''=\left(\int_{\mathfrak{b}_j}\omega_i\right),\quad
-2\eta'=\left(\int_{\mathfrak{a}_j}\eta_i\right), \quad
-2\eta''=\left(\int_{\mathfrak{b}_j}\eta_i\right).
\]
We define the period lattice $\Lambda=\big\{2\omega'm_1+2\omega''m_2\mid m_1,m_2\in\mathbb{Z}^2\big\}$ and consider the Jacobian variety $\operatorname{Jac}(V)=\mathbb{C}^2/\Lambda$.
The normalized period matrix is given by $\tau=(\omega')^{-1}\omega''$.
Let $\tau\delta'+\delta''$ with $\delta',\delta''\in\mathbb{R}^2$ be the Riemann constant with respect to $(\{\mathfrak{a}_i,\mathfrak{b}_i\}_{i=1}^2,\infty)$. 
Then we have $\delta'={}^t\big(\frac{1}{2},\frac{1}{2}\big)$ and $\delta''={}^t\big(1,\frac{1}{2}\big)$ (cf.~\cite{Fay-1973,Mumford-1983}). 
We denote the imaginary unit by ${\bf i}$. 
The sigma function $\sigma(u)$ associated with $V$, $u={}^t(u_1,u_3)\in\mathbb{C}^2$, is defined by
\begin{equation}
\sigma(u)=C\exp\bigg(\frac{1}{2}{}^tu\eta'(\omega')^{-1}u\bigg)\theta\begin{bmatrix}\delta'\\ \delta'' \end{bmatrix}\big((2\omega')^{-1}u,\tau\big),\label{2024.2.13.1}
\end{equation}
where $\theta\begin{bmatrix}\delta'\\ \delta'' \end{bmatrix}(u,\tau)$ is the Riemann theta function with the characteristics $\begin{bmatrix}\delta'\\ \delta'' \end{bmatrix}$ defined by
\[
\theta\begin{bmatrix}\delta'\\ \delta'' \end{bmatrix}(u,\tau)=\sum_{n\in\mathbb{Z}^2}\exp\big\{\pi{\bf i}\,{}^t(n+\delta')\tau(n+\delta')+2\pi{\bf i}\,{}^t(n+\delta')(u+\delta'')\big\}
\]
and $C$ is a non-zero constant which is fixed in Theorem \ref{2023.4.23.2}. 
Let $\sigma_i=\partial_{u_i}\sigma$, $\sigma_{ij}=\partial_{u_j}\sigma_i$, $\sigma_{ijk}=\partial_{u_k}\sigma_{ij}$, $\wp_{ij}=-\partial_{u_i}\partial_{u_j}\log\sigma$, $\wp_{ijk}=\partial_{u_k}\wp_{ij}$, and $\wp_{ijk\ell}=\partial_{u_{\ell}}\wp_{ijk}$, where $\partial_{u_i}=\frac{\partial}{\partial u_i}$. 

\begin{Proposition}[{\cite[pp.~7--8]{BEL-97-1}}]\label{period}
For $m_1,m_2\in\mathbb{Z}^2$, let $\Omega=2\omega'm_1+2\omega''m_2$.
Then, for $u\in\mathbb{C}^2$, we have
\[
\sigma(u+\Omega)/\sigma(u) =(-1)^{2({}^t\delta'm_1-{}^t\delta''m_2)+{}^tm_1m_2}\exp\big\{{}^t(2\eta'm_1+
2\eta''m_2)(u+\omega'm_1+\omega''m_2)\big\}.
\]
\end{Proposition}

Let $W=\big\{u\in\mathbb{C}^2\;|\;\sigma(u)=0\big\}$. 
Proposition \ref{period} implies that $u+\Omega\in W$ for any $u\in W$ and $\Omega\in \Lambda$. 
We set $\deg\lambda_{2i}=2i$ for $2\le i\le 5$.

\begin{Theorem}[{\cite[Theorem 6.3]{BEL-99-R}, \cite[Theorem 3]{N1}}]\label{2023.4.23.2}
The sigma function~$\sigma(u)$ is a holomorphic function on $\mathbb{C}^2$ and we have the unique constant $C$ in (\ref{2024.2.13.1}) such that the power series expansion of $\sigma(u)$ around the origin has the following form:
\begin{equation}
\sigma(u)=\frac{1}{3}u_1^3-u_3+\sum_{i+3j\ge7}\mu_{i,j}u_1^{i}u_3^{j},\label{4.27.1}
\end{equation}
where $\mu_{i,j}$ is a homogeneous polynomial in $\mathbb{Q}[\lambda_4, \lambda_6, \lambda_8, \lambda_{10}]$
of degree $i+3j-3$ if $\mu_{i,j}\neq0$.
\end{Theorem}

We take the constant $C$ in (\ref{2024.2.13.1}) such that the expansion (\ref{4.27.1}) holds. 
Then the sigma function $\sigma(u)$ does not depend on the choice of a canonical basis $\{\mathfrak{a}_i,\mathfrak{b}_i\}_{i=1}^2$ in the one-dimensional homology group of the curve $V$ and is determined only by the coefficients $\lambda_4$, $\lambda_6$, $\lambda_8$, $\lambda_{10}$ of the defining equation of the curve $V$.

\begin{Remark}
The constant $C$ is given explicitly in \cite[p.~906]{E-H-K-K-L-1} and \cite[p.~9]{E-H-K-K-L-P-1}.
\end{Remark}

\section{Inversion problem of ultra-elliptic integrals}\label{2024.2.15}

Let ${\boldsymbol \omega}={}^t(\omega_1,\omega_3)$. 
We consider the Abel--Jacobi map
\[I\colon\quad V\to\operatorname{Jac}(V),\qquad P\mapsto\int_{\infty}^P{\boldsymbol \omega}.\]

\begin{Lemma}[{Riemann vanishing theorem, e.g., \cite[Corollary 3.6]{Mumford-I}}]\label{2023.5.13.1}
We have 
\[\big\{u\in\operatorname{Jac}(V)\;|\;\sigma(u)=0\big\}=\big\{I(P)\;|\;P\in V\big\}.\] 
\end{Lemma}

It is well known that the inversion problem of the map $I$ can be solved in terms of the sigma function. 
Let $f_2(u)$, $g_2(u)$, $f_5(u)$, and $g_5(u)$ be the meromorphic functions on $\mathbb{C}^2$ defined by 
\begin{gather*}
f_2(u)=-\frac{\sigma_3(u)}{\sigma_1(u)},\qquad g_2(u)=\frac{\wp_{11}(2u)}{2},\\
f_5(u)=\frac{\sigma(2u)}{2\sigma_1(u)^4},\qquad g_5(u)=\frac{\sigma_1^2\sigma_{33}-2\sigma_1\sigma_3\sigma_{13}+\sigma_3^2\sigma_{11}}{2\sigma_1^3}(u).
\end{gather*}
For $P=(X,Y)\in V$, let $v=I(P)$. 

\begin{Lemma}[{\cite[p.~129]{G2}, \cite[Theorem 1]{J}}]\label{2023.3.4.2} As meromorphic functions on $V$, the following relation holds: 
\[X=f_2(v).\]
\end{Lemma}

\begin{Lemma}[{\cite[Lemma 3.4]{M}}]\label{2023.3.4.1} As meromorphic functions on $V$, the following relation holds: 
\[X=g_2(v).\]
\end{Lemma}

\begin{Lemma}[{\cite[p.~128]{G}, \cite[Lemma 3.2.4]{O-98}}]\label{2023.3.4.4} As meromorphic functions on $V$, the following relation holds: 
\[Y=f_5(v).\]
\end{Lemma}

\begin{Lemma}[{\cite[p.~116]{BEL-2012}, \cite[Lemma 3.2]{AB2019}}]\label{2023.3.4.3} As meromorphic functions on $V$, the following relation holds: 
\[Y=g_5(v).\]
\end{Lemma}

\begin{Lemma}[{Riemann singularity theorem, e.g., \cite[Theorem 3.16]{Mumford-I}}]\label{2023.4.22}
We have 
\[\big\{u\in W\;|\;\sigma_1(u)=\sigma_3(u)=0\big\}=\emptyset,\]
where $\emptyset$ is the empty set. 
\end{Lemma}

\begin{Lemma}\label{2023.3.24}
Let $F$ be a holomorphic function on $\mathbb{C}^2$. 
If $F$ is identically equal to $0$ on $W$, then there exists a holomorphic function $G$ on $\mathbb{C}^2$ such that 
\[F=\sigma G.\]
\end{Lemma}

\begin{proof}
Let $\mathcal{O}_{\mathbb{C}^2}$ be the sheaf of germs of holomorphic functions over $\mathbb{C}^2$ and $\mathcal{O}_{\mathbb{C}^2,w}$ be the stalk of $\mathcal{O}_{\mathbb{C}^2}$ at $w\in\mathbb{C}^2$. 
Then $\mathcal{O}_{\mathbb{C}^2,w}$ is a unique factorization domain for any $w\in\mathbb{C}^2$ (e.g., \cite[Theorem 1.12]{Kodaira}). 
For $w\in\mathbb{C}^2$, let $[\sigma]_w$ be the germ of $\sigma$ at $w$. 
From Lemma \ref{2023.4.22}, for any $w\in W$, $[\sigma]_w$ is irreducible in $\mathcal{O}_{\mathbb{C}^2,w}$. 
Therefore, for any $w\in\mathbb{C}^2$, there exist an open neighborhood $Q_w$ of $w$ in $\mathbb{C}^2$ and a holomorphic function $q_w$ on $Q_w$ such that 
$F=\sigma q_w$ on $Q_w$ (e.g., \cite[Theorem 1.16]{Kodaira}). 
Then $\{Q_w\}_{w\in\mathbb{C}^2}$ is an open covering of $\mathbb{C}^2$. 
We take $w_1, w_2\in\mathbb{C}^2$ such that $Q_{w_1}\cap Q_{w_2}\neq\emptyset$. 
If $u\in (Q_{w_1}\cap Q_{w_2})\backslash W$, we have 
\[q_{w_1}(u)=q_{w_2}(u)=\frac{F(u)}{\sigma(u)}.\]
Since $(Q_{w_1}\cap Q_{w_2})\backslash W$ is dense in $Q_{w_1}\cap Q_{w_2}$, we have $q_{w_1}(u)=q_{w_2}(u)$ for any $u\in Q_{w_1}\cap Q_{w_2}$. 
We can define the holomorphic function $G$ on $\mathbb{C}^2$ by $G(u)=q_w(u)$ for $u\in Q_w$. 
Then we have $F=\sigma G$. 
\end{proof}

\begin{Lemma}\label{2023.6.20.1}
Let $R$ be a Riemann surface. 
For $i=1,2$, let $F_i$ be a holomorphic function on $R$ which is not identically equal to $0$. 
Then $F_1F_2$ is not identically equal to $0$. 
\end{Lemma}

\begin{proof}
Although this result is well known, for the sake to be complete and self-contained, we give a proof of this result. 
There exists a point $r_0\in R$ such that $F_1(r_0)\neq0$. 
Since $F_1$ is continuous, there exists an open neighborhood $H$ of $r_0$ such that $F_1(r)\neq0$ for any $r\in H$.  
We assume that $F_1F_2$ is identically equal to $0$. 
Then we have $F_2(r)=0$ for any $r\in H$. 
By the identity theorem, we have $F_2(r)=0$ for any $r\in R$. (e.g., \cite[p.~6, Theorem 1.11]{Forster}). 
We arrive at a contradiction. Therefore, $F_1F_2$ is not identically equal to $0$. 
\end{proof}

\begin{Lemma}\label{2023.4.22.11}
For $i=1,2$, let $F_i$ be a holomorphic function on $\mathbb{C}^2$ which is not identically equal to $0$ on $W$. 
Then $F_1F_2$ is not identically equal to $0$ on $W$. 
\end{Lemma}

\begin{proof}
Let $\widetilde{V}$ be the universal covering of $V$ and $p : \widetilde{V}\to V$ be the projection. 
For $i=1,3$, let $\widetilde{\omega}_i$ be the pullback of $\omega_i$ with respect to $p$ and $\widetilde{\boldsymbol \omega}={}^t(\widetilde{\omega}_1,\widetilde{\omega}_3)$. 
We take a point $\widetilde{\infty}\in\widetilde{V}$ such that $p(\widetilde{\infty})=\infty$. 
We consider the map 
\[\widetilde{I}\colon\quad \widetilde{V}\to\mathbb{C}^2,\qquad P\mapsto\int_{\widetilde{\infty}}^P\widetilde{\boldsymbol \omega}.\]
Let $\kappa: \mathbb{C}^2\to\operatorname{Jac}(V)$ be the natural projection. 
Then we have the following commutative diagram:
\[
\xymatrix{
\widetilde{V} \ar[r]^-{\widetilde{I}} \ar[d]_-p & \mathbb{C}^2\ar[d]^-{\kappa} \\ 
V \ar[r]_-I & \operatorname{Jac}(V).
}
\]
From Lemma \ref{2023.5.13.1}, we have $W=\widetilde{I}(\widetilde{V})$. 
For $i=1,2$, let $\widetilde{F}_i=F_i\circ\widetilde{I}$. 
Then $\widetilde{F}_i$ is a holomorphic function on $\widetilde{V}$. 
Since $F_i$ is not identically equal to $0$ on $W$, $\widetilde{F}_i$ is not identically equal to $0$ on $\widetilde{V}$. 
From Lemma \ref{2023.6.20.1}, $\widetilde{F}_1\widetilde{F}_2$ is not identically equal to $0$ on $\widetilde{V}$. 
Therefore, $F_1F_2$ is not identically equal to $0$ on $W$. 
\end{proof}

\begin{Lemma}\label{2023.4.23.1}
Let $F$ be a meromorphic function on $\mathbb{C}^2$. 
We can express $F$ in the form of 
\begin{equation}
F=\sigma^m\frac{G_1}{G_2},\label{2024.2.21.1}
\end{equation}
where $m$ is an integer and $G_i$, $i=1,2$, are holomorphic functions on $\mathbb{C}^2$ which are not identically equal to $0$ on $W$. 
In particular, $m$ is uniquely determined by $F$. 
\end{Lemma}

\begin{proof}
By the Poincar\'e theorem, there exist holomorphic functions $F_1$ and $F_2$ on $\mathbb{C}^2$ such that $F=F_1/F_2$ (e.g., \cite[p.~336, Theorem 5.9]{Freitag}). 
From Lemma \ref{2023.3.24}, for $i=1,2$, there exist a non-negative integer $m_i$ and a holomorphic function $G_i$ such that $G_i$ is not identically equal to $0$ on $W$ and $F_i=\sigma^{m_i}G_i$. 
Therefore, we have 
\[F=\sigma^{m_1-m_2}\frac{G_1}{G_2}.\]
Let $m=m_1-m_2$. 
Then we obtain (\ref{2024.2.21.1}). 
Assume that we have another expression of $F$ in the form of (\ref{2024.2.21.1}) 
\[F=\sigma^n\frac{H_1}{H_2},\]
where $n$ is an integer and $H_i$, $i=1,2$, are holomorphic functions on $\mathbb{C}^2$ which are not identically equal to $0$ on $W$. 
Without loss of generality, we can assume $m\ge n$. 
Then we have 
\[\sigma^{m-n}G_1H_2=G_2H_1.\]
From Lemma \ref{2023.4.22.11}, $G_2H_1$ is not identically equal to $0$ on $W$. 
Therefore, we have $m=n$. 
\end{proof}


\begin{Proposition}
As meromorphic functions on $\mathbb{C}^2$, we have $f_5\neq g_5$. 
\end{Proposition}

\begin{proof}
We consider the following holomorphic functions on $\mathbb{C}^2$:  
\begin{gather*}
f_{5,n}(u)=\sigma(2u),\qquad f_{5,d}(u)=2\sigma_1(u)^4,\\
g_{5,n}(u)=(\sigma_1^2\sigma_{33}-2\sigma_1\sigma_3\sigma_{13}+\sigma_3^2\sigma_{11})(u),\qquad g_{5,d}(u)=2\sigma_1(u)^3.
\end{gather*}
Then we have $f_5=f_{5,n}/f_{5,d}$ and $g_5=g_{5,n}/g_{5,d}$. 
From Theorem \ref{2023.4.23.2}, the power series expansion of $f_{5,n}g_{5,d}-f_{5,d}g_{5,n}$ around the origin has the following form: 
\[(f_{5,n}g_{5,d}-f_{5,d}g_{5,n})(u)=\frac{4}{3}u_1^9-4u_1^6u_3+\sum_{i+3j\ge13}\nu_{i,j}u_1^iu_3^j,\]
where $\nu_{i,j}$ is a homogeneous polynomial in $\mathbb{Q}[\lambda_4, \lambda_6, \lambda_8, \lambda_{10}]$
of degree $i+3j-9$ if $\nu_{i,j}\neq0$.
Thus, the holomorphic function $f_{5,n}g_{5,d}-f_{5,d}g_{5,n}$ is not identically equal to $0$ on $\mathbb{C}^2$. 
Therefore, we obtain the statement of the proposition. 
\end{proof}

From Lemmas \ref{2023.5.13.1}, \ref{2023.3.4.2}, and \ref{2023.4.22}, we have 
\begin{equation}
\big\{u\in W\;|\;\sigma_1(u)=0\big\}=\Lambda.\label{2023.6.10.111}
\end{equation}
The values of the meromorphic functions $f_5$ and $g_5$ are defined on $W\backslash\Lambda$. 
From Lemmas \ref{2023.3.4.4} and \ref{2023.3.4.3}, the values of $f_5$ and $g_5$ coincide on $W\backslash\Lambda$. 
The set $W\backslash\Lambda$ is dense in $W$. 
Therefore, from Lemma \ref{2023.3.24}, the function $f_5-g_5$ can be divided by the sigma function at least one time. 

\begin{Proposition}[{\cite[Lemma 7.1]{AB2019}}]\label{2023.4.23.44}
We have the following decomposition: 
\[f_5-g_5=\sigma \frac{h}{2\sigma_1^4},\]
where 
\[h=\sigma_1(\sigma_3\sigma_{113}-\sigma_1\sigma_{133}-\sigma_{11}\sigma_{33}+2\sigma_{13}^2)-2\sigma_3^3-\sigma_3\sigma_{11}\sigma_{13}+\sigma(\sigma_{11}\sigma_{133}-\sigma_{13}\sigma_{113}+3\sigma_3\sigma_{33})-\sigma^2\sigma_{333}.\]
\end{Proposition}

From (\ref{2023.6.10.111}), we find that $2\sigma_1^4$ is not identically equal to $0$ on $W$. 
From Theorem \ref{2023.4.23.2}, we have $\sigma(0)=\sigma_1(0)=\sigma_{11}(0)=0$ and $\sigma_3(0)=-1$. 
Thus, we have $h(0)=2$. 
Therefore, $h$ is not identically equal to $0$ on $W$. 

\begin{Proposition}
As meromorphic functions on $\mathbb{C}^2$, we have $f_2\neq g_2$. 
\end{Proposition}

\begin{proof}
We have 
\[\wp_{11}(2u)=\frac{\sigma_1(2u)^2-\sigma(2u)\sigma_{11}(2u)}{\sigma(2u)^2}.\]
We consider the following holomorphic functions on $\mathbb{C}^2$: 
\begin{gather*}
f_{2,n}(u)=-\sigma_3(u),\qquad f_{2,d}(u)=\sigma_1(u),\\
g_{2,n}(u)=\sigma_1(2u)^2-\sigma(2u)\sigma_{11}(2u),\qquad g_{2,d}(u)=2\sigma(2u)^2.
\end{gather*}
Then we have $f_2=f_{2,n}/f_{2,d}$ and $g_2=g_{2,n}/g_{2,d}$. 
From Theorem \ref{2023.4.23.2}, the power series expansion of $f_{2,n}g_{2,d}-f_{2,d}g_{2,n}$ around the origin has the following form: 
\[(f_{2,n}g_{2,d}-f_{2,d}g_{2,n})(u)=\frac{80}{9}u_1^6-\frac{88}{3}u_1^3u_3+8u_3^2+\sum_{i+3j\ge10}\xi_{i,j}u_1^iu_3^j,\]
where $\xi_{i,j}$ is a homogeneous polynomial in $\mathbb{Q}[\lambda_4, \lambda_6, \lambda_8, \lambda_{10}]$
of degree $i+3j-6$ if $\xi_{i,j}\neq0$.
Thus, the holomorphic function $f_{2,n}g_{2,d}-f_{2,d}g_{2,n}$ is not identically equal to $0$ on $\mathbb{C}^2$. 
Therefore, we obtain the statement of the proposition. 
\end{proof}

The value of the meromorphic function $f_2$ is defined on $W\backslash\Lambda$. 
Let $\{\gamma_i\}_{i=1}^5$ be the set of roots of the polynomial $M(X)$. 
For $1\le i\le 5$, let $\Lambda_i$ be the preimage of $I\big((\gamma_i,0)\big)$ under the projection $\kappa: \mathbb{C}^2\to\operatorname{Jac}(V)$. 
Let $\Lambda_0=\Lambda$. 
The value of the meromorphic function $g_2$ is defined on $W\backslash(\bigcup_{i=0}^5\Lambda_i)$. 
From Lemmas \ref{2023.3.4.2} and \ref{2023.3.4.1}, the values of $f_2$ and $g_2$ coincide on $W\backslash(\bigcup_{i=0}^5\Lambda_i)$. 
The set $W\backslash(\bigcup_{i=0}^5\Lambda_i)$ is dense in $W$. 
Therefore, from Lemma \ref{2023.3.24}, the function $f_2-g_2$ can be divided by the sigma function at least one time. 
Let $\sigma_{\operatorname{rat}}$ be the polynomial obtained by substituting $\lambda_4=\lambda_6=\lambda_8=\lambda_{10}=0$ into the right hand side of (\ref{4.27.1}).  
Then we have $\sigma_{\operatorname{rat}}=\frac{1}{3}u_1^3-u_3$. 

\begin{Proposition}[{\cite[Lemma 7.2]{AB2019}}]
In the rational limit, i.e., $\lambda_4=\lambda_6=\lambda_8=\lambda_{10}=0$, we have the following decomposition: 
\[f_2-g_2=\sigma_{\operatorname{rat}}\frac{3(10u_1^3-3u_3)}{u_1^2(4u_1^3-3u_3)^2}.\]
\end{Proposition}

The purpose of this paper is to decompose $f_2-g_2$ into a product of the sigma function and a meromorphic function on $\mathbb{C}^2$ explicitly in the general case.  

\begin{Lemma}\label{2023.7.21.333}
Let $x=\wp_{11}(u)$, $y=\wp_{13}(u)$, and $z=\wp_{33}(u)$. 
The function $\wp_{11}(2u)$ is expressed in terms of $x$, $y$, and $z$ as follows: 
\begin{equation}
\wp_{11}(2u)=\frac{c_0+\lambda_4c_4+\lambda_6c_6+\lambda_8c_8+\lambda_{10}c_{10}+\sum_{i,j,k,\ell}\lambda_4^i\lambda_6^j\lambda_8^k\lambda_{10}^{\ell}a_{ijk\ell}}
{d_0+\lambda_4d_4+\lambda_6d_6+\lambda_8d_8+\lambda_{10}d_{10}+\sum_{i,j,k,\ell}\lambda_4^i\lambda_6^j\lambda_8^k\lambda_{10}^{\ell}b_{ijk\ell}},\label{2023.7.21.1}
\end{equation}
where 
\begin{align*}
c_0&=x^2y^4(5xz-3y^2),\quad c_4=x^4y^2(3xz-y^2),\quad c_6=-2x^6yz,\\
c_8&=x^6(xz+y^2),\quad c_{10}=-2x^7y,\\
d_0&=-x^3y^5,\quad d_4=-x^5y^3,\quad d_6=x^6y^2,\quad d_8=-x^7y,\quad d_{10}=x^8,\\
a_{0000}&=4y^5(xz-y^2),\\
a_{1000}&=y^3(4y+3x^2)(2xz-3y^2),\\
a_{0100}&=2y^2(24y^2z+6x^2yz-2x^4z+6xy^3+5x^3y^2),\\
a_{0010}&=-y^2(56xyz+22x^3z+76y^3+41x^2y^2+8x^4y),\\
a_{2000}&=y^2(20xyz+9x^3z+4y^3-x^2y^2-x^4y),\\
a_{0001}&=2(40y^3z+50x^2y^2z+20x^4yz+3x^6z+70xy^4+30x^3y^3+2x^5y^2),\\
a_{1100}&=2y(2y+x^2)(4yz-3x^2z-6xy^2+x^3y),\\
a_{1010}&=8xy^2z+6x^3yz+3x^5z-88y^4-50x^2y^3-10x^4y^2-3x^6y,\\
a_{0200}&=-4y(4xyz+2x^3z-16y^3-12x^2y^2-3x^4y),\\
a_{3000}&=y^3(12y+5x^2),\\
a_{1001}&=2(24y^2z+18x^2yz+5x^4z+60xy^3+44x^3y^2+11x^5y+2x^7),\\
a_{0110}&=-4(4y^2z+x^2yz-x^4z+14xy^3+10x^3y^2+3x^5y),\\
a_{2100}&=-2xy^2(2y+x^2),\\
a_{0101}&=4(4xyz+3x^3z+36y^3+27x^2y^2+8x^4y+2x^6),\\
a_{0020}&=4xyz+5x^3z-76y^3-41x^2y^2-8x^4y+x^6,\\
a_{2010}&=y(4y^2-x^2y-x^4),\\
a_{0011}&=4(4yz+5x^2z+18xy^2+11x^3y+4x^5),\\ 
a_{2001}&=4x(3y^2+3x^2y+x^4),\\
a_{1110}&=-4xy(2y+x^2),\\
a_{0002}&=4(4xz+16y^2+12x^2y+5x^4), \\ 
a_{1101}&=4(y+x^2)(4y+3x^2),\\
a_{1020}&=-12y^2-9x^2y-x^4,\\
a_{1011}&=4x(2y+3x^2),\\ 
a_{0201}&=8x(2y+x^2),\\
a_{0120}&=-2x(2y+x^2),\\
a_{1002}&=16x^2, \\ 
a_{0111}&=4(4y+3x^2),\\ 
a_{0030}&=-4y-3x^2,\\
a_{0102}&=16x, \\ 
a_{0021}&=-4x, \\ 
b_{0000}&=-y^4(36yz+5x^2z+6xy^2),\\
b_{1000}&=-y^2(24y^2z+26x^2yz+3x^4z+26xy^3+10x^3y^2),\\
b_{0100}&=xy(24y^2z+18x^2yz+2x^4z+39xy^3+12x^3y^2),\\
b_{0010}&=24y^3z-2x^2y^2z-8x^4yz-x^6z-46xy^4-58x^3y^3-14x^5y^2,\\
b_{2000}&=-y^2(4yz-3x^2z+2xy^2+x^3y),\\
b_{0001}&=-4(20xy^2z+10x^3yz+x^5z-15y^4-30x^2y^3-20x^4y^2-4x^6y),\\
b_{1100}&=2xy(4yz-x^2z-5xy^2),\\
b_{1010}&=8y^2z-2x^2yz+x^4z+4xy^3+14x^3y^2+x^5y,\\
b_{0200}&=-4x^2y(z-xy),\\
b_{3000}&=2xy^3,\\
b_{1001}&=2(2x^3z+28y^3+8x^2y^2-8x^4y-x^6),\\
b_{0110}&=-2x(4yz-x^2z-3xy^2+2x^3y),\\
b_{2100}&=-x^2y^2,\\
b_{0101}&=4x(2xz-22y^2-14x^2y-x^4),\\
b_{0020}&=-4yz+3x^2z+22xy^2+7x^3y+2x^5,\\
b_{2010}&=2xy^2,\\
b_{0011}&=8(2xz-y^2-x^2y+x^4),\\
b_{2001}&=12y^2+x^4,\\
b_{1110}&=-2x^2y,\\
b_{0002}&=16(z+xy+x^3),\\
b_{1101}&=-4x(2y-x^2),\\
b_{1020}&=-2xy,\\
b_{1011}&=-8(y-x^2),\\
b_{0201}&=4x^2,\\
b_{0120}&=-x^2,\\
b_{1002}&=16x,\\
b_{0111}&=8x,\\
b_{0030}&=-2x,\\
b_{0102}&=16,\\
b_{0021}&=-4,\\
\end{align*}
and $a_{ijk\ell}=b_{ijk\ell}=0$ otherwise. 
Here, we arrange $a_{ijk\ell}$ and $b_{ijk\ell}$ in ascending order of the value $4i+6j+8k+10\ell$. 
Note that $c_i$ and $d_i$ are the homogeneous polynomials in $x$, $y$, and $z$ of degree $8$. 
Further, $a_{ijk\ell}$ and $b_{ijk\ell}$ are the polynomials in $x$, $y$, and $z$ of degree at most $7$. 
The exponents on $z$ appearing in $c_i$, $d_i$, $a_{ijk\ell}$, and $b_{ijk\ell}$ are at most $1$. 
\end{Lemma}

\begin{proof}
In \cite{A-E-E-2004}, \cite[pp.~39, 48]{Baker-1907}, \cite[Corollaries 3.1.2, 3.1.3, Theorem 3.2]{BEL-97-1}, and \cite[pp.~228--229]{BEL-2012}, the following formulae are given: 
\begin{align}
\wp_{111}^2&=4\wp_{33}+4\lambda_4\wp_{11}+4\wp_{11}^3+4\wp_{13}\wp_{11}+4\lambda_6,\label{5.22.1}\\
\wp_{111}\wp_{113}&=2\lambda_8+2\wp_{13}^2-2\wp_{33}\wp_{11}+2\lambda_4\wp_{13}+4\wp_{13}\wp_{11}^2,\label{5.22.2}\\
\wp_{113}^2&=4\lambda_{10}-4\wp_{33}\wp_{13}+4\wp_{11}\wp_{13}^2,\label{5.22.3}\\
\wp_{133}&=\wp_{111}\wp_{13}-\wp_{11}\wp_{113},\label{3.23.112}\\
\wp_{333}&=2\wp_{11}\wp_{133}-\wp_{33}\wp_{111}-\wp_{13}\wp_{113}-\lambda_4\wp_{113}, \label{3.23.113}\\
\wp_{1111}&=6\wp_{11}^2+4\wp_{13}+2\lambda_4,\label{7.3.1}\\
\wp_{1113}&=6\wp_{11}\wp_{13}-2\wp_{33}.\label{7.3.2}
\end{align}
In \cite[p.~129]{Baker-1907} and \cite[Proposition 4.10]{uchida}, for $u\in\mathbb{C}^2$, the following formula is given:  
\begin{equation}
\wp_{11}(2u)=\wp_{11}(u)+\frac{\phi_1(u)^2-\phi(u)\phi_{11}(u)}{4\phi(u)^2},\label{7.3.6}
\end{equation}
where 
\[\phi(u)=\wp_{11}(u)\wp_{133}(u)-\wp_{13}(u)\wp_{113}(u)+\wp_{333}(u),\]
$\phi_1=\partial_{u_1}\phi$, and $\phi_{11}=\partial_{u_1}\phi_1$.  
From (\ref{3.23.112}) and (\ref{3.23.113}), we can express the function $\phi$ in terms of $\wp_{11}, \wp_{13}, \wp_{33}, \wp_{111}$, and $\wp_{113}$ as follows: 
\begin{equation}
\phi=3\wp_{11}\wp_{13}\wp_{111}-\wp_{33}\wp_{111}-2\wp_{13}\wp_{113}-3\wp_{11}^2\wp_{113}-\lambda_4\wp_{113}.\label{7.3.3}
\end{equation}
By differentiating both sides of (\ref{7.3.3}) with respect to $u_1$ and using (\ref{5.22.1})--(\ref{3.23.112}), (\ref{7.3.1}), and (\ref{7.3.2}), we can express the function $\phi_1$ in terms of $\wp_{11}, \wp_{13}$, and $\wp_{33}$ as follows: 
\begin{equation}
\phi_1=4(\wp_{11}^2\wp_{33}-\wp_{11}\wp_{13}^2+4\wp_{13}\wp_{33}+\lambda_4\wp_{11}\wp_{13}+2\lambda_6\wp_{13}-\lambda_8\wp_{11}-2\lambda_{10}).\label{7.3.4}
\end{equation}
By differentiating both sides of (\ref{7.3.4}) with respect to $u_1$ and using (\ref{3.23.112}), we can express the function $\phi_{11}$ in terms of $\wp_{11}, \wp_{13}, \wp_{33}, \wp_{111}$, and $\wp_{113}$ as follows: 
\begin{align}
\phi_{11}&=4(\wp_{11}^2\wp_{13}\wp_{111}-\wp_{11}^3\wp_{113}-6\wp_{11}\wp_{13}\wp_{113}+3\wp_{13}^2\wp_{111}+2\wp_{11}\wp_{33}\wp_{111}+4\wp_{33}\wp_{113}\nonumber\\
&+\lambda_4\wp_{11}\wp_{113}+\lambda_4\wp_{13}\wp_{111}+2\lambda_6\wp_{113}-\lambda_8\wp_{111}).\label{7.3.5}
\end{align}
In \cite[p.~186]{Hudson-1990}, \cite[p.~38]{Baker-1907}, and \cite[p.~229]{BEL-2012}, the following relation between $\wp_{11}, \wp_{13}$, and $\wp_{33}$ is given: 
\begin{equation}
(\wp_{11}^2+4\wp_{13})\wp_{33}^2=2S\wp_{33}-T,\label{2023.7.19.1}
\end{equation}
where 
\begin{align*}
S&=\wp_{11}\wp_{13}^2-\lambda_4\wp_{11}\wp_{13}-2\lambda_6\wp_{13}+\lambda_8\wp_{11}+2\lambda_{10},\\
T&=\wp_{13}^4+2\lambda_4\wp_{13}^3-4\lambda_6\wp_{11}\wp_{13}^2+2\lambda_8\wp_{13}(\wp_{13}+2\wp_{11}^2)-4\lambda_{10}\wp_{11}(\wp_{11}^2+\wp_{13})+\lambda_4^2\wp_{13}^2\\
&+2\lambda_4\lambda_8\wp_{13}-4\lambda_4\lambda_{10}\wp_{11}-4\lambda_6\lambda_{10}+\lambda_8^2.
\end{align*}
From (\ref{7.3.6})--(\ref{7.3.5}) and (\ref{5.22.1})--(\ref{5.22.3}), we can express $\wp_{11}(2u)$ in terms of $\wp_{11}(u), \wp_{13}(u)$, and $\wp_{33}(u)$. 
By using (\ref{2023.7.19.1}), we eliminate $\wp_{33}(u)^k$ with $k\ge2$ from this expression of $\wp_{11}(2u)$. 
Then we obtain (\ref{2023.7.21.1}). 
\end{proof}

\begin{Remark}
In \cite[p.~130]{Baker-1907}, it is mentioned that $\wp_{11}(2u)$ can be expressed in terms of $\wp_{11}(u)$, $\wp_{13}(u)$, and $\wp_{33}(u)$. 
In \cite{Baker-1907}, $\wp_{11}(2u)$ is not expressed in terms of $\wp_{11}(u)$, $\wp_{13}(u)$, and $\wp_{33}(u)$ explicitly. 
\end{Remark}

\begin{Remark}
In Lemma \ref{2023.7.21.333}, the author used the computer algebra system Maxima for computation. 
\end{Remark}

\begin{Theorem}\label{2023.4.27.1}
We have the following decomposition: 
\begin{equation}
f_2-g_2=\sigma \frac{A}{B},\label{2023.7.21.2}
\end{equation}
where 
\begin{align*}
A&=\alpha_0+\lambda_4\alpha_4+\lambda_6\alpha_6+\lambda_8\alpha_8+\lambda_{10}\alpha_{10}+\sigma\sum_{i,j,k,\ell}\lambda_4^i\lambda_6^j\lambda_8^k\lambda_{10}^{\ell}\sigma^{14}(\sigma_1a_{ijk\ell}+2\sigma_3b_{ijk\ell}),\\
B&=-2\sigma_1\left(\beta_0+\lambda_4\beta_4+\lambda_6\beta_6+\lambda_8\beta_8+\lambda_{10}\beta_{10}+\sigma^2\sum_{i,j,k,\ell}\lambda_4^i\lambda_6^j\lambda_8^k\lambda_{10}^{\ell}\sigma^{14}b_{ijk\ell}\right),\\
\alpha_0&=(\sigma_1^2-\sigma\sigma_{11})^2(\sigma_1\sigma_3-\sigma\sigma_{13})^4\\
&\quad \times\left\{\sigma_1(8\sigma_1\sigma_3\sigma_{13}-5\sigma_1^2\sigma_{33}-3\sigma_3^2\sigma_{11})+\sigma(5\sigma_1\sigma_{11}\sigma_{33}-2\sigma_3\sigma_{11}\sigma_{13}-3\sigma_1\sigma_{13}^2)\right\},\\
\alpha_4&=(\sigma_1^2-\sigma\sigma_{11})^4(\sigma_1\sigma_3-\sigma\sigma_{13})^2\\
&\quad \times\left\{\sigma_1(4\sigma_1\sigma_3\sigma_{13}-3\sigma_1^2\sigma_{33}-\sigma_3^2\sigma_{11})+\sigma(3\sigma_1\sigma_{11}\sigma_{33}-2\sigma_3\sigma_{11}\sigma_{13}-\sigma_1\sigma_{13}^2)\right\},\\
\alpha_6&=2(\sigma_1^2-\sigma\sigma_{11})^6(\sigma_1\sigma_3-\sigma\sigma_{13})(\sigma_1\sigma_{33}-\sigma_3\sigma_{13}),\\
\alpha_8&=(\sigma_1^2-\sigma\sigma_{11})^6\left\{\sigma_1(\sigma_3^2\sigma_{11}-\sigma_1^2\sigma_{33})+\sigma(\sigma_1\sigma_{11}\sigma_{33}-2\sigma_3\sigma_{11}\sigma_{13}+\sigma_1\sigma_{13}^2)\right\},\\
\alpha_{10}&=2(\sigma_1^2-\sigma\sigma_{11})^7(\sigma_1\sigma_{13}-\sigma_3\sigma_{11}),\quad \beta_0=-(\sigma_1^2-\sigma\sigma_{11})^3(\sigma_1\sigma_3-\sigma\sigma_{13})^5,\\
\beta_4&=-(\sigma_1^2-\sigma\sigma_{11})^5(\sigma_1\sigma_3-\sigma\sigma_{13})^3,\quad \beta_6=(\sigma_1^2-\sigma\sigma_{11})^6(\sigma_1\sigma_3-\sigma\sigma_{13})^2,\\ 
\beta_8&=-(\sigma_1^2-\sigma\sigma_{11})^7(\sigma_1\sigma_3-\sigma\sigma_{13}),\quad \beta_{10}=(\sigma_1^2-\sigma\sigma_{11})^8. 
\end{align*}
Here, $a_{ijk\ell}$ and $b_{ijk\ell}$ are defined in Lemma \ref{2023.7.21.333}. 
Note that $\sigma^{14}a_{ijk\ell}$ and $\sigma^{14}b_{ijk\ell}$ can be expressed as polynomials in $\sigma, \sigma_1, \sigma_3, \sigma_{11}, \sigma_{13}$, and $\sigma_{33}$.  
The holomorphic functions $A$ and $B$ are not identically equal to $0$ on $W$. 
\end{Theorem}

\begin{proof}
We substitute 
\[
\wp_{11}=\frac{\sigma_1^2-\sigma\sigma_{11}}{\sigma^2},\;\;\;\;\wp_{13}=\frac{\sigma_1\sigma_3-\sigma\sigma_{13}}{\sigma^2},\;\;\;\;\wp_{33}=\frac{\sigma_3^2-\sigma\sigma_{33}}{\sigma^2}
\]
into the right hand side of (\ref{2023.7.21.1}) and multiply the numerator and the denominator of the right hand side of (\ref{2023.7.21.1}) by $\sigma^{16}$. 
Then we can draw $\sigma$ from $f_2-g_2$ and obtain (\ref{2023.7.21.2}). 
We can express $A$ in the form of   
\[A=\sigma_1^9A_1+\sigma A_2,\]
where 
\begin{align*}
A_1&=\sigma_3^4(8\sigma_1\sigma_3\sigma_{13}-5\sigma_1^2\sigma_{33}-3\sigma_3^2\sigma_{11})+\lambda_4\sigma_1^2\sigma_3^2(4\sigma_1\sigma_3\sigma_{13}-3\sigma_1^2\sigma_{33}-\sigma_3^2\sigma_{11})\\
&+2\lambda_6\sigma_1^4\sigma_3(\sigma_1\sigma_{33}-\sigma_3\sigma_{13})+\lambda_8\sigma_1^4(\sigma_3^2\sigma_{11}-\sigma_1^2\sigma_{33})+2\lambda_{10}\sigma_1^5(\sigma_1\sigma_{13}-\sigma_3\sigma_{11})
\end{align*}
and $A_2$ is a holomorphic function on $\mathbb{C}^2$. 
The power series expansion of $A_1$ around the origin has the following form: 
\[A_1(u)=-6u_1+\sum_{i+3j\ge5}\zeta_{i,j}u_1^iu_3^j,\]
where $\zeta_{i,j}$ is a homogeneous polynomial in $\mathbb{Q}[\lambda_4, \lambda_6, \lambda_8, \lambda_{10}]$
of degree $i+3j-1$ if $\zeta_{i,j}\neq0$. 
From Theorem \ref{2023.4.23.2}, it is impossible to decompose $A_1$ into a product of the sigma function and a holomorphic function on $\mathbb{C}^2$. 
Therefore, from Lemma \ref{2023.3.24}, $A_1$ is not identically equal to $0$ on $W$. 
Since $\sigma_1$ is not identically equal to $0$ on $W$, the function $A$ is not identically equal to $0$ on $W$. 
We can express $B$ in the form of   
\[B=2\sigma_1^{12}B_1+\sigma B_2,\]
where 
\[B_1=\sigma_3^5+\lambda_4\sigma_1^2\sigma_3^3-\lambda_6\sigma_1^3\sigma_3^2+\lambda_8\sigma_1^4\sigma_3-\lambda_{10}\sigma_1^5\]
and $B_2$ is a holomorphic function on $\mathbb{C}^2$. 
From $\sigma_1(0)=0$ and $\sigma_3(0)=-1$, we have $B_1(0)=-1$. 
Thus, $B_1$ is not identically equal to $0$ on $W$. Since $\sigma_1$ is not identically equal to $0$ on $W$, the function $B$ is not identically equal to $0$ on $W$. 
\end{proof}

\vspace{2ex}

{\bf Acknowledgements.} This work was supported by JSPS KAKENHI Grant Number JP21K03296 and was partly supported by MEXT Promotion of Distinctive Joint Research Center Program JPMXP0723833165.

\end{document}